\documentclass{amsart}

\usepackage{amsmath, amssymb, amsthm,enumitem,bm,pict2e,xcolor,url}
\usepackage{indentfirst}

\theoremstyle{definition}
\newtheorem{theorem}{Theorem}
\newtheorem*{theorem*}{Theorem}
\numberwithin{theorem}{section}
\newtheorem{definition}[theorem]{Definiton}
\newtheorem{proposition}[theorem]{Proposition}
\newtheorem{lemma}[theorem]{Lemma}
\newtheorem{remark}[theorem]{Remark}

\newtheorem{cor}[theorem]{Corollary}

\DeclareMathOperator{\grad}{grad}
\DeclareMathOperator{\spanned}{span}

\DeclareMathOperator{\id}{id}
\DeclareMathOperator{\Sym}{Sym}

\DeclareMathOperator{\initial}{in}

\DeclareMathOperator{\Lie}{Lie}
\DeclareMathOperator{\wt}{wt}

\newcommand{\la}{\lambda}
\newcommand{\om}{\omega}
\newcommand{\mU}{\mathcal U}
\newcommand{\fg}{\mathfrak{g}}
\newcommand{\fsl}{\mathfrak{sl}}
\newcommand{\fn}{\mathfrak{n}}
\newcommand{\fh}{\mathfrak{h}}
\newcommand{\fb}{\mathfrak{b}}
\newcommand{\bC}{\mathbb{C}}
\newcommand{\bP}{\mathbb{P}}

\newcommand{\bZ}{\mathbb{Z}}

\mathcode`l="8000
\begingroup
\makeatletter
\lccode`\~=`\l
\DeclareMathSymbol{\lsb@l}{\mathalpha}{letters}{`l}
\lowercase{\gdef~{\ifnum\the\mathgroup=\m@ne \ell \else \lsb@l \fi}}%
\endgroup

\title{PBW degenerate Schubert varieties: Cartan components and counterexamples}
\author{I. Makhlin}
\address{Skolkovo Institute of Science and Technology\\ 
Center for Advanced Studies\\
Bolshoy Boulevard 30, bld. 1\\
Moscow 121205\\
Russia}
\address{National Research University Higher School of Economics\\
International Laboratory of Representation Theory and Mathematical Physics\\
Ulitsa Usacheva 6\\Moscow 119048\\Russia\vspace{2ex}}
\email{imakhlin@mail.ru}

\thanks{The author would like to thank Lara Bossinger, Xin Fang, Evgeny Feigin, Ghislain Fourier and Ievgen Makedonskyi for helpful discussions of these subjects. The work was partially supported by the grant RSF 19-11-00056. This research was supported in part by Young Russian Mathematics award.}

\begin{document}

\maketitle

\begin{abstract}
In recent years PBW degenerations of Demazure modules and Schubert varieties were defined and studied in several papers. Various interesting properties (such as these PBW degenerations embedding naturally into the corresponding degenerate representations and flag varieties) were obtained in type $\mathrm A$ but only with restrictions on the Weyl group element or the highest weight. We show that these properties cannot hold in full generality due to the following issue with the definition. The degenerate variety depends on the highest weight used to define it and not only on its Weyl group stabilizer (as is the case for PBW degenerate flag varieties as well as classical Schubert varieties). Perhaps surprisingly, the minimal counterexamples appear only for $\fsl_6$. The counterexamples are constructed with the help of a study of the Cartan components appearing in this context.
\end{abstract}

\section*{Introduction}

Over the last decade PBW degenerations of representations and flag varieties proved to be a diverse and fruitful research topic (\cite{Fe,FFL1,ABS,FFL2,CFR,CL} and many others). Let us briefly outline the definitions of these objects.

Consider a semisimple Lie algebra $\fg$ and the subalgebra $\fn_-\subset\fg$ spanned by negative root vectors, an integral dominant $\fg$-weight $\la$ and the irreducible representation $L_\la$. The PBW filtration on $\mU(\fn_-)$ defines a filtration on $L_\la$ via action on the highest weight vector, the associated graded space $L_\la^a$ is the PBW degenerate representation (of the abelian Lie algebra $\fn_-^a$). The space $L_\la^a$ is then seen to be acted upon by the abelian Lie group $N_-^a=\bC^{\dim\fn_-}$, the closure $F_\la^a$ of the $N_-^a$-orbit of the highest weight point in $\bP(L_\la^a)$ is the corresponding PBW degenerate flag variety. 

Let us restrict our attention to algebras $\fg$ of type $\mathrm A$. Here, a fundamental property of the degenerate flag varieties which makes them especially interesting is that they depend on the highest weight $\la$ the same way as flag varieties do themselves. Let $\la=a_1\om_1+\ldots+a_{n-1}\om_{n-1}$ for fundamental weights $\om_i$, then $F_\la^a$ is determined by the Weyl group stabilizer of $\la$, i.e.\ the set of such $i$ that $a_i>0$. This was first proved in~\cite{Fe} by constructing a degenerate Pl\"ucker embedding of sorts for $F^a_\la$. A highly important fact, from which the Pl\"ucker embedding is obtained, is that $L_\la^a$ can be realized as the cyclic submodule generated by the highest weight vector in \[(L_{\omega_1}^a)^{\otimes a_1}\otimes\ldots\otimes (L_{\omega_{n-1}}^a)^{\otimes a_{n-1}},\] i.e.\ the so-called Cartan component. 

Now, a very natural generalization here seems to be the transition from irreducible representations and flag varieties to Demazure modules and Schubert varieties. For a Weyl group element $w$ the Demazure module $D_{w\la}$ is generated by a lowest weight vector. This allows us to define its PBW degeneration $D_{w\la}^a$ and the orbit closure $X_{w\la}^a\subset\bP(D_{w\la}^a)$ in analogy with the above. 

These objects were, apparently, first considered in~\cite{fourier} where many of the key properties of $L_\la^a$ and $F_\la^a$ were generalized to the case of a \emph{triangular} $w$, i.e.\ a permutation avoiding the patterns $4231$ and $2413$. It was shown that for such $w$ the variety $X_{w\la}^a$ admits a degenerate Pl\"ucker embedding and, consequently, depends on $\la$ the same way the Schubert (or flag) variety does. The mentioned Cartan component property was also generalized to this case and a generalization of the FFLV basis (\cite{FFL1}) was constructed. Let us point out that, in a sense, the underlying cause for this neat picture is the fact that the PBW filtration on $D_{w\la}$ can be induced from the one on $L_\la$. This realizes $D_{w\la}^a$ as the a cyclic submodule in $L_\la^a$ over a certain subalgebra in $\fn_-^a$ and also realizes $X_{w\la}^a$ as an orbit closure in $F_\la^a$ for certain subgroup in $N_-^a$ (see also Remark~\ref{birational}).

In the subsequent paper~\cite{BF} similar properties are obtained in the case of a minuscule highest weight and arbitrary $w$. The most recent paper~\cite{CFF} on this subject proved that when $w$ is \emph{rectangular} (a subclass of the triangular elements), $D_{w\la}^a$ is itself a Demazure module and $X_{w\la}^a$ is itself a Schubert variety for some algebra $\fsl_m$ of higher rank. This generalizes analogous properties of $L_\la^a$ and $F_\la^a$ proved in~\cite{CL}. 

Little, however, has been proved concerning type $\mathrm A$ PBW degenerate Demazure modules and Schubert varieties of full generality. Nevertheless, some of the experts in this field expected the degenerate Pl\"ucker embedding or, at least, the mentioned independence of $\la$ to generalize to the case of an arbitrary $w$ (as was communicated to the author during some private and public discussions). We show that this is not the case proving that neither of these results hold in full generality. Perhaps the most surprising aspect here is that the minimal counterexamples appear only for $\fsl_6$ which makes spotting them a challenge even with the use of a computer. As discussed in Remark~\ref{n45}, for $n\le 5$ the Cartan component property has been verified for all $w$ which implies both the Pl\"ucker embedding and the independence of $\la$. This circumstance lead the author to put a considerable amount of effort into proving the same for greater $n$ only to discover said counterexamples after gradually optimizing his search algorithms.

Prior to presenting the counterexamples in Section~\ref{counterex} we provide a study of the Cartan component $E_{w\la}$ and the orbit closure subvariety in $\bP(E_{w\la})$. It turns out that these objects are somewhat easier to control than $D_{w\la}^a$ and $X_{w\la}^a$ but can be tied to the latter via a homomorphism between $D_{w\la}^a$ and $E_{w\la}$. These results are then applied to the construction of counterexamples. It should be pointed out that the proofs in Section~\ref{counterex} require some computer assistance (the Demazure modules in question have dimensions 2942 and 8226). This, however, is kept to a bare minimum, i.e.\ finding Minkowski sums of certain small sets: a computation which is easily reproducible.

\section{Preliminaries}

We first recall the definitions and key properties of the more conventional abelian PBW degenerations of $\fsl_n$-representations and flag varieties and then define the analogous objects for Demazure modules and Schubert varieties.

\subsection{PBW degenerate representations and flag varieties}\label{abelianPBW}

For $n\ge 2$ consider the Lie algebra $\fg=\fsl_n(\bC)$ with a fixed triangular decomposition $\fg=\fn_-\oplus\fh\oplus\fn_+$. For $1\le k\le n-1$ denote the simple roots $\alpha_k\in\fh^*$  and let $\om_k\in\fh^*$ be the corresponding fundamental weights. The set of positive roots $\Phi^+$ is comprised of \[\alpha_{i,j}=\alpha_i+\ldots+\alpha_{j-1}\] for $1\le i<j\le n-1$ (the set of negative roots is denoted $\Phi^-$). Let $\fn_-$ be spanned by negative root vectors $f_{i,j}$ with weight $-\alpha_{i,j}$ and $\fn_+$ be spanned by positive root vectors $e_{i,j}$ with weight $\alpha_{i,j}$. Our basis of choice in $\fh^*$ will be the set of fundamental weights, i.e.\ $(a_1,\ldots,a_{n-1})$ will denote the weight $a_1\om_1+\ldots+a_{n-1}\om_{n-1}$. The semigroup of integral dominant weights (those with coordinates in $\bZ_{\ge 0}$) will be denoted $\Pi^+$.

The universal enveloping algebra $\mU(\fn_-)$ is equipped with a $\bZ_{\ge 0}$-filtration (the PBW filtration) with components \[\mU(\fn_-)_{\le m}=\spanned(f_{i_1,j_1}\ldots f_{i_N,j_N}|N\le m).\] The associated $\bZ_{\ge 0}$-graded algebra $\mU^a$ is a polynomial algebra in $n\choose 2$ variables $f_{i,j}^a$ (the images of $f_{i,j}$ in $\mU(\fn_-)_{\le 1}/\mU(\fn_-)_{\le 0}$). The algebra $\mU^a$ can be viewed as $\mU(\fn_-^a)$ for an abelian Lie algebra $\fn_-^a$ spanned by the $f_{i,j}^a$.

For $\la\in\Pi^+$ let $L_\la$ be the irreducible representation of $\fg$ with highest weight $\la$ and highest weight vector $v_\la$ (so that $\fn_+v_\la=0$). The PBW filtration induces a $\bZ_{\ge 0}$-filtration on every such $L_\la$ by $(L_\la)_{\le m}=\mU(\fn_-)_{\le m} v_\la$. The associated graded space $L_\la^a$ is naturally a $\mU^a$-module known as the (abelian) PBW degeneration of $L_\la$. This module is cyclic, generated by a vector $v_\la^a$ (the image of $v_\la$ in $(L_\la)_{\le 0}$).

Now consider the Lie group $G=SL_n(\bC)$ with $\Lie(G)=\fg$ and the nilpotent subgroup $N_-\subset G$ with $\Lie(N_-)=\fn_-$. Choose $\la=(a_1,\ldots,a_{n-1})\in\Pi^+$ with $d=d(\la)=(d_1,\ldots,d_l)$ the tuple of integers for which $a_{d_k}>0$. The tuple $d$ determines the parabolic subgroup $P\subset G$ preserving the line $\bC v_\la\subset L_\la$. This line provides a point $\bm v_\la$ in the projectivization $\bP(L_\la)$. The stabilizer $G_{\bm v_\la}$ is seen to be $P$, therefore the orbit closure $\overline{N_-\bm v_\la}$ is the partial flag variety $F_d=G/P$ (with the orbit itself being the largest Bruhat cell). 

Similarly, we may consider the connected simply connected Lie group $N_-^a$ with $\Lie(N_-^a)=\fn_-^a$ (this Lie group will simply be $\bC^{n\choose2}$ under addition). The group $N_-^a$ acts on $L_\la^a$ and $\bP(L_\la^a)$. We have $\bm v_\la^a\in\bP(L_\la^a)$ corresponding to $\bC v_\la^a$ and we define $F_d^a=\overline{N_-\bm v_\la^a}$ --- the (abelian) PBW degeneration of $F_d$. The fact that $F_d^a$ depends only on $d$ and not on $\la$ itself is a consequence of Theorem~\ref{degenplucker} (the degenerate Pl\"ucker embedding) which we now state.

Let the $n$-dimensional complex space $V$ be the standard representation of $\fg=\fsl_n$ with basis $b_1,\ldots,b_n$. The irreducible representations with fundamental highest weights can be explicitly described as $L_{\om_k}=\wedge^k V$ with a basis consisting of the vectors \[b_{i_1,\ldots,i_k}=b_{i_1}\wedge\ldots\wedge b_{i_k}.\] We may assume that $v_{\om_k}=b_{1,\ldots,k}$.

Consider the Pl\"ucker embedding \[F_d\subset\bP_d=\bP(L_{\om_{d_1}})\times\ldots\times\bP(L_{\om_{d_l}}).\] The product $\bP_d$ is equipped with the Pl\"ucker coordinates $X_{i_1,\ldots,i_{d_k}}$ with $1\le k\le l, 1\le i_1<\ldots<i_{d_k}\le n$, coordinate $X_{i_1,\ldots,i_{d_k}}$ corresponding to $b_{i_1,\ldots,i_{d_k}}\in L_{\om_{d_k}}$. The homogeneous coordinate ring of $\bP_d$ is $R_d=\bC[\{X_{i_1,\ldots,i_{d_k}}\}]$. The homogeneous coordinate ring of $F_d$ is then $\mathcal P_d=R_d/I_d$, where $I_d$ is the ideal of Pl\"ucker relations.

Now we introduce a grading on $R_d$ by setting 
\begin{equation}\label{grada}
\grad^a(X_{i_1,\ldots,i_{d_k}})=s_{i_1,\ldots,i_{d_k}}=|\{j|i_j>d_k\}|.
\end{equation}
Let us note that this grading has a straightforward interpretation: $s_{i_1,\ldots,i_{d_k}}$ is the least $m$ such that $b_{i_1,\ldots,i_{d_k}}\in (L_{\om_{d_k}})_{\le m}$. For instance, we have $s_{2,4,5}=2$ which is the PBW degree of the monomial $f_{1,5}f_{3,4}$ mapping $v_{\om_3}=b_{1,2,3}$ to $b_{2,4,5}$.

We proceed to consider the corresponding initial ideal $\initial_{\grad^a}I_d$, i.e.\ the ideal spanned by components of minimal grading of elements of $I_d$.
\begin{theorem}[{\cite[Theorem 3.13]{Fe}}]\label{degenplucker}
$F_d^a$ is isomorphic to the subvariety in $P_d$ cut out by the ideal $\initial_{\grad^a}I_d$.
\end{theorem}

We will also require another (closely related) fact from this theory.
\begin{definition}
Given a finite set $\{M_i\}$ of cyclic modules over a Hopf algebra with respective generators $\{v_i\}$, the submodule in $\bigotimes M_i$ generated by $\bigotimes v_i$ is known as the \emph{Cartan component}.
\end{definition}
Let us consider two instances of this very general notion. For the chosen $\la$ consider the tensor product \[U_\la=L_{\omega_1}^{\otimes a_1}\otimes\ldots\otimes L_{\omega_{n-1}}^{\otimes a_{n-1}}.\] It is well known that the Cartan component in $U_\la$ generated by \[u_\la=v_{\omega_1}^{\otimes a_1}\otimes\ldots\otimes v_{\omega_{n-1}}^{\otimes a_{n-1}}\] is isomorphic to the irreducible representation $L_\la$ (simply because $u_\la$ has weight $\la$ and is the unique highest weight vector in $U_\la$). A similar fact holds for PBW degenerations. Denote \[U_\la^a=(L_{\omega_1}^a)^{\otimes a_1}\otimes\ldots\otimes (L_{\omega_{n-1}}^a)^{\otimes a_{n-1}}\] and \[u_\la^a=(v_{\omega_1}^a)^{\otimes a_1}\otimes\ldots\otimes (v_{\omega_{n-1}}^a)^{\otimes a_{n-1}}.\]
\begin{theorem}[essentially due to \cite{FFL1} but see also {\cite[Theorem 10.4]{favourable}}]\label{cartancomp}
The Cartan component $\mU(\fn_-^a)u_\la^a\subset U_\la^a$ is isomorphic to $L_\la^a$.
\end{theorem}

Next, let us explain the nature of the embedding $F_d^a\subset \bP_{d}$. On one hand, in view of Theorem~\ref{cartancomp}, $F_{d}^a$ is isomorphic to $\overline{N_-^a\bm u_\la^a}\subset\bP(U_\la^a)$ where $\bm u_\la^a$ corresponds to the line $\bC u_\la^a$. On the other, we have the embeddings \[\bP_{d}\subset\bP(L_{\om_{d_1}})^{a_{d_1}}\times\ldots\times\bP(L_{\om_{d_l}})^{a_{d_l}}\subset\bP(U_\la^a)\] where the first embedding is diagonal and the second is the Segre embedding. The image of $\bP_d$ under this embedding contains $\bm u_\la^a$ and this image is preserved by the $N_-^a$-action on $\bP(U_\la^a)$, therefore the orbit $N_-^a\bm u_\la^a$ and its closure are contained in $\bP_d\subset\bP(U_\la^a)$.

This lets us characterize the ideal $\initial_{\grad^a}I_{d}$ as the kernel of a certain map from $R_{d}$ to a polynomial ring. Choose $c_{i,j}\in\bC$ for $1\le i<j\le n$ and consider \[f=\sum c_{i,j}f_{i,j}^a\in\fn_-^a.\] 
\begin{definition}
For $d_k$ and a tuple $1\le i_1<\ldots<i_{d_k}\le n$ consider the vector $\exp(f)v_{\om_{d_k}}^a\in L_{\om_{d_k}}^a$ and decompose it into a linear combination of the $b_{i_1,\ldots,i_{d_k}}^a$. Denote $p_{i_1,\ldots,i_{d_k}}\in\mathbb Q[\{z_{i,j}, 1\le i<j\le n\}]$ the polynomial (homogeneous of degree $s_{i_1,\ldots,i_{d_k}}$) such that the coefficient of $b_{i_1,\ldots,i_{d_k}}^a$ is equal to $p_{i_1,\ldots,i_{d_k}}(\{c_{i,j}\})$. 
\end{definition}
One sees that such a polynomial exists by expanding $\exp(f)v_{\om_{d_k}}^a$ into a (terminating) series.

\begin{proposition}[see {\cite[Proof of Theorem 3.4]{fafefom}}]\label{idealaskernel}
The ideal $\initial_{\grad^a}I_{d}$ is the kernel of the homomorphism from $R_{d}$ to $\bC[\{z_{i,j},1\le i<j\le n\},\{z_k,1\le k\le n-1\}]$ taking $X_{i_1,\ldots,i_{d_k}}$ to $z_{d_k}p_{i_1,\ldots,i_{d_k}}$.
\end{proposition}


\subsection{PBW degenerate Demazure modules and Schubert varieties}

The Weyl group $W\cong S_n$ acts on $\fh^*$, the permutation $w\in W$ maps $\alpha_{i,j}$ to $\alpha_{w(i),w(j)}$ where we set $\alpha_{k,l}=-\alpha_{l,k}$ if $k>l$. For every $w\in W$ there is a unique (up to a scalar multiple) vector $v_{w\la}\in L_\la$ of weight $w\la$. 
\begin{definition}
The $\mU(\fn_+)$-module $D_{w\la}=\mU(\fn_+)v_{w_\la}$ is the corresponding Demazure module. 
\end{definition}
\begin{definition}
Let $\fn_{+w}\subset\fn_+$ be the subalgebra spanned by $e_{i,j}$ with $w^{-1}(\alpha_{i,j})\in\Phi^-$ (i.e.\ $w^{-1}(i)>w^{-1}(j)$).
\end{definition}
Note that for every positive root $\alpha_{i,j}$ such that $w^{-1}(\alpha_{i,j})\in\Phi^+$ we have $e_{i,j}v_{w\la}=0$. 
Consequently, $D_{w\la}=\mU(\fn_{+w})v_{w_\la}$. The dimension of $\fn_{+w}$ is the permutation length $l(w)$.

Let $\bm v_{w\la}\in\bP(L_\la)$ correspond to $\bC v_{w\la}$ and $N_+=\exp(\fn_+)\subset G$ while $N_{+w}=\exp(\fn_{+w})\subset N_+$. The Schubert variety $X_{w,d(\la)}$ is the orbit closure $\overline{N_+\bm v_{w\la}}=\overline{N_{+w}\bm v_{w\la}}$. This orbit closure indeed depends only on $w$ and $d(\la)$ and is naturally embedded into $F_{d(\la)}$ since $\bm v_{w\la}$ is contained in $F_{d(\la)}\subset\bP(L_\la)$.

Now, like above, we have the PBW filtration on $\mU(\fn_{+w})$ with component $\mU(\fn_{+w})_{\le m}$ spanned by PBW monomials of degree no greater than $m$. This induces the filtration $(D_{w\la})_{\le m}=\mU(\fn_{+w})_{\le m}v_{w\la}$. 
\begin{definition}
The associated graded algebra for the former filtration is $\mU(\fn_{+w}^a)$ for an abelian Lie algebra $\fn_{+w}^a$. The associated graded space for the latter filtration is $D_{w\la}^a$, the PBW degeneration of the Demazure module. 
\end{definition}
The algebra $\fn_{+w}^a$ is spanned by $e_{i,j}^a$ which are images of $e_{i,j}\in\fn_{+w}$. The space $D_{w\la}^a$ is a module over $\mU(\fn_{+w}^a)$. Furthermore, $D_{w\la}^a$ is acted upon by the corresponding Lie group $N_{+w}^a$ (i.e. $\bC^{l(w)}$ under addition) and so is $\bP(D_{w\la}^a)$. Let $D_{w\la}^a$ be generated by $v_{w\la}^a$, let the point $\bm v_{w\la}^a\in\bP(D_{w\la}^a)$ correspond to $\bC v_{w\la}^a$.
\begin{definition}
$X_{w\la}^a=\overline{N_{+w}^a\bm v_{w\la}^a}$ is the PBW degeneration of the Schubert variety $X_{w,d(\la)}$. 
\end{definition}
It should be noted that we could consider the PBW degeneration $\fn_+^a$ of the whole $\fn_+$, view $D_{w\la}^a$ as a $\fn_+^a$-module and define $X_{w\la}^a$ as a $N_+^a$-orbit closure for the larger group $N_+^a\cong\bC^{n\choose2}$. However, all $e_{i,j}^a\notin\fn_{+w}^a$ act trivially on $D_{w\la}^a$ as well as the corresponding subgroups $\exp(\bC e_{i,j}^a)\subset N_+^a$. This is why we lose nothing by limiting our attention to $\fn_{+w}$.

An observation due to~\cite[Subsection 1.2]{fourier} is that it is useful to consider a shift of this construction by $w^{-1}$. Recall that the group $W$ acts on $\fg$ by automorphisms that preserve $\fh$ while permuting the root subspaces. This induces a $W$-action on every $L_\la$. 
\begin{definition}
Consider the subalgebra $\fn_{-w}=w^{-1}(\fn_{+w})\subset\fg$. 
\end{definition}
By the definition of $\fn_{+w}$, the subalgebra $\fn_{-w}$ is contained in $\fn_-$ and is spanned by such $f_{i,j}$ that $w(i)>w(j)$. In fact, we have $\fn_{-w}=\fn_-\cap w^{-1}(\fn_+)$ while $\fn_{+w}=\fn_+\cap w(\fn_-)$. The space \[\widetilde D_{w\la}=w^{-1}(D_{w\la})=\mU(\fn_{-w})v_\la\] will also be a Demazure module but with respect to a different Borel subalgebra: not $\fb_+=\fn_+\oplus\fh$ but $w^{-1}(\fb_+)$. 


Once again, we have the PBW filtrations on $\mU(\fn_{-w})_{\le m}$ and $(\widetilde D_{w\la})_{\le m}=\mU(\fn_{-w})_{\le m}v_\la$. 
\begin{definition}
Let $\mU(\fn_{-w}^a)$ and $\widetilde D_{w\la}^a$ (PBW degeneration of $\widetilde D_{w\la}$) be the respective associated graded objects.
\end{definition}
We can naturally identify $\fn_{-w}^a$ with the subalgebra in $\fn_-^a$ spanned by $f_{i,j}^a$ with $w(i)>w(j)$. The space $\widetilde D_{w\la}^a$ is acted upon by $\mU(\fn_{-w}^a)$ and the corresponding Lie group $N_{-w}^a\cong \bC^{l(w)}$. Let $\widetilde{\bm v}_\la^a\in\bP(\widetilde D_{w\la}^a)$ correspond to the line $\bC v_\la^a$.
\begin{proposition}
The orbit closure $\overline{N_{-w}^a\widetilde{\bm v}_{\la}^a}$ is isomorphic to $X_{w\la}^a$.
\end{proposition}
\begin{proof}
Consider the isomorphism $\psi:\fn_{+w}^a\to\fn_{-w}^a$ mapping $e_{i,j}^a$ to $f_{w^{-1}(j),w^{-1}(i)}^a$. We can now view $\widetilde D_{w\la}^a$ as a $\fn_{+w}^a$-module by letting $e_{i,j}^a$ act as $\psi(e_{i,j}^a)$. It is evident that $\widetilde D_{w\la}^a$ and $D_{w\la}^a$ are isomorphic $\fn_{+w}^a$-modules. Furthermore, we have the isomorphism $\exp(\psi):N_{+w}^a\to N_{-w}^a$ which defines an $N_{+w}^a$-action on $\widetilde D_{w\la}^a$. We see that $D_{w\la}^a$ and $\widetilde D_{w\la}^a$ are isomorphic $N_{+w}^a$-representations and the proposition follows.
\end{proof}

To summarize the Proposition and its proof, the objects $\fn_{-w}^a$, $\widetilde D_{w\la}^a$ and $\overline{N_{-w}^a\widetilde{\bm v}_{\la}^a}$ are isomorphic to the objects $\fn_{+w}^a$, $D_{w\la}^a$ and $X_{w\la}^a$ and can be studied instead of these latter objects. The usefulness comes from the fact that $\fn_{-w}^a$ (resp. $N_{-w}^a$) is naturally embedded into $\fn_-^a$ (resp. $N_-^a$) which lets us exploit our understanding of PBW degenerations of irreducible representations and flag varieties. This idea is applied in~\cite[Section 5]{fourier} to show, among other things, that if $w$ is sufficiently nice (``triangular''), then $X_{w\la}^a$ depends only on $d(\la)$ and embeds naturally into $F_{d(\la)}^a$. This approach is then furthered in~\cite[Theorem 2.4]{CFF} to show that for a certain narrower class of permutations $w$ (termed ``rectangular'') the variety $X_{w\la}^a$ is actually a Schubert variety for a certain $SL_m$. 

\section{General observations}\label{general}

In this section we will derive some general facts about the PBW degenerations $\widetilde D_{w\la}^a$ and $X_{w\la}^a$ via the study of the corresponding Cartan components. These facts will be useful to us in the construction of our counterexamples.

Within this section fix $w\in W$ and $\la=(a_1,\ldots,a_{n-1})\in\Pi^+$ with $d(\la)=d=(d_1,\ldots,d_l)$. Since $\widetilde D_{w\la}^a$ and $L_{\la}^a$ are defined as associated graded spaces, they are equipped with a $\bZ_{\ge 0}$-grading that we also denote $\grad^a$. For a $\grad^a$-graded space $A$ we will simply write $A_m$ to denote its $\grad^a$-homogeneous component of grading $m$. Also observe that, since $\fn_{-w}^a$ is embedded into $\fn_-^a$, we have a $\fn_{-w}^a$-module structure on $L_\la^a$.
\begin{proposition}
There exists a homomorphism of $\fn_{-w}^a$-modules \[\varphi_{w\la}:\widetilde D_{w\la}^a\to L_{\la}^a\] that respects $\grad^a$ and takes $\widetilde v_{w\la}^a$ to $v_\la^a$.
\end{proposition}
\begin{proof}
The Demazure module $\widetilde D_{w\la}$ is defined as a subspace of $L_{\la}$ which provides an embedding $(\widetilde D_{w\la})_{\le m}\subset(L_\la)_{\le m}$ (PBW filtrations with respect to $\fn_{-w}$ and $\fn_-$ respectively) for all $m$. These embeddings induce maps $(\widetilde D_{w\la}^a)_m\to(L_\la^a)_m$ which sum up to give $\varphi_{w\la}$.
\end{proof}

\begin{definition} 
Denote $E_{w\la}\subset L_\la^a$ the image of $\varphi_{w\la}$. 
\end{definition}
In other words, $E_{w\la}=\mU(\fn_{-w}^a)v_\la^a$. 
\begin{definition}
Denote $Y_{w,d}$ the variety $\overline{N_{-w}^a\bm v_\la^a}\subset \bP(E_{w\la})\subset\bP(L_\la^a)$. 
\end{definition}
Since $N_{-w}^a$ is embedded in to $N_-^a$, the variety $Y_{w,d}$ is embedded into $F_d^a$ and depends only on $d$, hence the notation. The modules $\widetilde D_{w\om_{k}}^a$ and $E_{w\om_k}$ are easily described.
\begin{proposition}\label{fundamental}
The maps $\varphi_{w\om_k}$ are injective for all $1\le k\le n-1$ and $w\in W$. The subspace $E_{w\om_k}\subset L_{\om_k}^a$ is spanned by the vectors $b_{i_1,\ldots,i_k}^a$ for which the following holds. For any $1\le l\le k$ the $l$th smallest number among $w(i_1),\ldots,w(i_k)$ is no greater than the $l$th smallest number among $w(1),\ldots,w(k)$.
\end{proposition}
\begin{proof}
The vectors $b_{i_1,\ldots,i_k}$ with the property from the last sentence of the proposition are known to span $\widetilde D_{w\om_k}$. Therefore, for dimensional reasons, one is only to verify that each such $b_{i_1,\ldots,i_k}^a$ is indeed contained in $E_{w\om_k}$. Such a $b_{i_1,\ldots,i_k}^a$ is obtained from $v_{\om_k}^a$ by the action of the product of all $f_{i,j}^a$ such that $w(i)$ is the $l$th smallest number among $w(1),\ldots,w(k)$ and $w(j)$ is the $l$th smallest number among $w(i_1),\ldots,w(i_k)$ for some $1\le l\le k$ and, moreover, $i<j$.
\end{proof}

\begin{remark}\label{remcartcomp}
The module $\widetilde D_{w\om_k}^a=E_{w\om_k}$ is embedded into $L_{\om_k}^a$. Therefore $E_{w\la}$ is the $\mU(\fn_{-w}^a)$-submodule in the tensor product \[(\widetilde D_{w\omega_1}^a)^{\otimes a_1}\otimes\ldots\otimes (\widetilde D_{\omega_{n-1}}^a)^{\otimes a_{n-1}}\] generated by the highest weight vector. This lets one view $E_{w\la}$ as the Cartan component for PBW degenerations of Demazure modules.
\end{remark}

\begin{definition}
Let $\bm\varphi_{w\la}:\bP(\widetilde D_{w\la}^a)\to \bP(L_\la^a)$ be the rational map induced by $\varphi_{w\la}$. 
\end{definition}

\begin{proposition}\label{biratprop}
The map $\bm\varphi_{w\la}$ restricts to a birational equivalence between $X_{w\la}^a$ and $Y_{w,d}$. 
\end{proposition}
\begin{proof} 
For $f_{i,j}^a\in\fn_{-w}^a$ the vector $f_{i,j}^a v_\la^a$ is the image of $f_{i,j}v_\la$ in $(L_\la)_{\le 1}/(L_\la)_{\le 0}\subset L_\la^a$, hence $f_{i,j}^a v_\la^a=0$ if and only if $f_{i,j} v_\la=0$. Similarly, $f_{i,j}^a \widetilde v_\la^a=0$ if and only if $f_{i,j} v_\la=0$, therefore $\widetilde{\bm v}_{\la}^a$ and ${\bm v}_{\la}^a$ have the same stabilizer in $N_{-w}^a$. Since $\varphi_{w\la}$ is a homomorphism of $N_{-w}^a$-modules, the map $\bm\varphi_{w\la}$ identifies the orbits $N_{-w}^a\widetilde{\bm v}_{\la}^a$ and $N_{-w}^a\bm v_\la^a$ and thus provides a birational equivalence between their closures.
\end{proof}

\begin{remark}\label{birational}
It is shown in in~\cite[Subsection 3.4]{fourier} that if $w$ is triangular, then the map $\varphi_{w\la}$ is injective. The map $\bm\varphi_{w\la}$, therefore, is an injective morphism and provides an isomorphism between $X_{w\la}^a$ and $Y_{w,d}$. Part of our goal in this paper is showing that this does not necessarily happen for general $w$ and we may, moreover, have $X_{w\la}^a\not\cong Y_{w,d}$.
\end{remark}

The defining ideal of $Y_{w,d}$ can be characterized algebraically.
\begin{definition}
Let $J_{w,d}\subset R_d$ be the ideal cutting out $Y_{w,d}\subset P_d$. 
\end{definition}
\begin{definition}
Denote $p^w_{i_1,\ldots,i_{d_k}}$ the polynomial obtained from the polynomial $p_{i_1,\ldots,i_{d_k}}$ defined in Subsection~\ref{abelianPBW} by setting $z_{i,j}=0$ when $w(i)<w(j)$.
\end{definition}
\begin{proposition}\label{idealaskernel1}
The ideal $J_{w,d}$ is the kernel of the homomorphism from $R_d$ to $\bC[\{z_{i,j}\},\{z_k,1\le k\le n-1\}]$ taking $X_{i_1,\ldots,i_{d_k}}$ to $z_{d_k}p^w_{i_1,\ldots,i_{d_k}}$.
\end{proposition}
\begin{proof}
Choose numbers $c_{i,j}\in\bC$ for all pairs $1\le i<j\le n$ with $w(i)>w(j)$ (i.e.\ $f_{i,j}^a\in\fn_{-w}^a$) and consider $f=\sum c_{i,j}f_{i,j}^a\in\fn_{-w}^a$. Choose some $d_k$ and a tuple $1\le i_1<\ldots<i_{d_k}\le n$, consider the vector $\exp(f)v_{\om_{d_k}}^a\in L_{\om_{d_k}}^a$ and decompose it into a linear combination of the $b_{i_1,\ldots,i_{d_k}}^a$. The coefficient of $b_{i_1,\ldots,i_{d_k}}^a$ will equal $p^w_{i_1,\ldots,i_{d_k}}(\{c_{i,j}\})$. This (together with the bijectivity of the exponential map for $\fn_{-w}^a$) implies the proposition.
\end{proof}

Before proceeding we observe that the ring $R_d$ is graded by $\Pi^+$ by setting $\deg(X_{i_1,\ldots,i_{d_k}})=\om_{d_k}$. The defining ideals of subvarieties in $P_d$ are $\deg$-homogeneous and so are their initial ideals. For a $\deg$-graded space $A$ we will write $A_\mu$ to denote its $\deg$-homogeneous component of grading $\mu$. 

Now let us give a slight modification of Theorem~\ref{cartancomp}. Denote \[T_\la^a=\Sym^{a_1}(L_{\omega_1}^a)\otimes\ldots\otimes \Sym^{a_{n-1}}(L_{\omega_{n-1}}^a)\] and \[t_\la^a=(v_{\omega_1}^a)^{a_1}\otimes\ldots\otimes (v_{\omega_{n-1}}^a)^{a_{n-1}}\in T_\la^a.\] Since $T_\la^a$ is embedded into $U_\la$ as the space of tensors that are symmetric in the according sense, we immediately see that the $\fn_-^a$-submodule $\mU(\fn_-^a)t_\la^a\subset T_\la^a$ is isomorphic to $L_\la^a$.

The space $T_\la^a$ has a basis comprised of vectors obtained by taking for each $k$ a (symmetric) product of $a_k$ vectors of the form $e_{i_1,\ldots,i_k}^a$ and then taking the tensor product of these $n-1$ symmetric products. The component $(R_d)_\la$ has a basis obtained from this basis in $T_\la^a$ by replacing each $e_{i_1,\ldots,i_k}^a$ with $X_{i_1,\ldots,i_k}$ and symmetric and tensor products by products in the ring $R_d$. The bijection between these two bases establishes a duality between $T_\la^a$ and $(R_d)_\la$.
\begin{proposition}\label{edual}
The subspace $(J_{w,d})_\la\subset (R_d)_\la$ is the orthogonal of the subspace $E_{w\la}\subset T_\la^a$ with respect to the above duality.
\end{proposition}
\begin{proof}
We have identified $(R_d)_\la$ with the space of linear forms on $\bP(T_\la^a)$, i.e.\ $H^0(\bP(T_\la^a),\mathcal O_{\bP(T_\la^a)}(1))$. Note that the image of the embedding $\bP_d\subset\bP(U_\la^a)$ is contained in $\bP(T_\la^a)\subset\bP(U_\la^a)$. By the definition of the Segre embedding, $p\in (R_d)_\la$ vanishes as a linear form in a point in $x\in\bP_d\subset\bP(T_\la^a)$, if and only if $p$ vanishes as polynomial in the Pl\"ucker coordinates of $x$.

Let $E$ be the subspace in $T_\la^a$ dual to $(J_{w,d})_\la$. Then $\bP(E)$ is the minimal projective subspace in $\bP(T_\la^a)$ which contains $Y_{w,d}$ but this is precisely $\bP(E_{w\la})$ because $E_{w\la}$ is cyclic.
\end{proof}

Next, the subgroup $N_{-w}=\exp(\fn_{-w})\subset N_-$ defines a Schubert variety $\widetilde X_{w,d}=\overline{N_{-w}v_\la}\subset F_d$, this subvariety is isomorphic to $X_{w,d}$ but is embedded differently into $F_d$ and, subsequently, $P_d$. The embedding $\widetilde X_{w,d}\subset P_d$ provides an ideal $\widetilde I_{w,d}\subset R_d$.

Also note that the PBW filtration on $L_\la$ induces a filtration on $\widetilde D_{w\la}$, the associated graded space $H_{w\la}$ is a $\fn_{-w}^a$-submodule in $L_\la^a$ that contains $E_{w\la}$.
\begin{proposition}\label{hdual}
The subspace $(\initial_{\grad^a}\widetilde I_{w,d})_\la\subset (R_d)_\la$ is the orthogonal of the subspace $H_{w\la}\subset T_\la^a$.
\end{proposition}
\begin{proof}
We may consider the space \[T_\la=\Sym^{a_1}(L_{\omega_1})\otimes\ldots\otimes \Sym^{a_{n-1}}(L_{\omega_{n-1}})\] with the linear isomorphism $\tau:T_\la^a\to T_\la$  which identifies the corresponding products of the $b_{i_1,\ldots,i_k}^a$ and the $b_{i_1,\ldots,i_k}$. We obtain a $\grad^a$-grading on $T_\la$ via $(T_\la)_m=\tau((T_\la^a)_m)$ and a filtration $(T_\la)_{\le m}=\bigoplus_{l\le m}(T_\la)_l$. We see that $f_{i,j}$ maps $(T_\la)_{\le m}$ to $(T_\la)_{\le m+1}$, therefore the associated graded space $\mathrm{gr}T_\la$ is a $\fn_-^a$-module. By comparing the action of $f_{i,j}^a$ on $\mathrm{gr}T_\la$ and $T_\la^a$, we see that the two are isomorphic as $\fn_-^a$-modules. The fact that the highest weight vector in $T_\la$ generates $L_\la$ while that in $T_\la^a$ generates $L_\la^a$ implies that $L_\la\cap(T_\la)_{\le m}=(L_\la)_{\le m}$.

From the above we see that the subspace $L_\la^a\subset T_\la^a$ is spanned by all vectors of the form $\tau^{-1}(\initial_{\grad^a}(v))$ where $v\in L_\la\subset T_\la$ and $\initial_{\grad^a}(v)$ is the nonzero $\grad^a$-homogeneous component of $v$ with the \emph{largest} grading. Consequently, $H_{w\la}\subset L_\la^a\subset T_\la^a$ is spanned by $\initial_{\grad^a}(v)$ with $v\in\widetilde D_{w\la}\subset L_\la\subset T_\la$.

The above description of $H_{w\la}$ together with the fact that $\widetilde D_{w\la}\subset T_\la$ is dual to $(\widetilde I_{w,d})_\la\subset(R_d)_\la$ implies that $H_{w\la}$ and $(\initial_{\grad^a}\widetilde I_{w,d})$ do indeed annihilate each other. The proposition then follows for dimensional reasons.
\end{proof}

We can now embed $Y_{w,d}$ into a Gr\"obner degeneration of $\widetilde X_{w,d}$ (i.e. a variety cut out by an initial ideal of $\widetilde I_{w,d}$).
\begin{proposition}\label{injiff}
We have $\initial_{\grad^a}\widetilde I_{w,d}\subset J_{w,d}$, i.e.\ $Y_{w,d}$ is a subvariety of the variety cut out by $\initial_{\grad^a}\widetilde I_{w,d}$.
The map $\varphi_{w\la}$ is injective if and only if $(\initial_{\grad^a}\widetilde I_{w,d})_\la=(J_{w,d})_\la$. 
\end{proposition}
\begin{proof}
The first claim is immediate from Propositions~\ref{edual} and~\ref{hdual} in view of $E_{w\la}\subset H_{w\la}$. Now, the map $\varphi_{w\la}$ is injective if and only if $\dim E_{w\la}=\dim D_{w\la}$. However, $\dim H_{w\la}=\dim D_{w\la}$ and Propositions~\ref{edual} and~\ref{hdual} imply \[\dim(\initial_{\grad^a}\widetilde I_{w,d})_\la+\dim H_{w\la}=\dim(J_{w,d})_\la+\dim E_{w\la}.\] The second claim follows. 
\end{proof}

Now we can prove the following.
\begin{cor}\label{tfae}
The following are equivalent.
\begin{enumerate}[label=(\alph*)]
\item The map $\varphi_{w\mu}$ is injective for all $\mu\in\Pi^+$ with $d(\mu)=d$.
\item $\initial_{\grad^a}\widetilde I_{w,d}=J_{w,d}$.
\item $\initial_{\grad^a}\widetilde I_{w,d}$ is prime.
\end{enumerate}
\end{cor}
\begin{proof}
By the second part of Proposition~\ref{injiff}, (a) is equivalent to (b). Since $Y_{w,d}$ is irreducible (as an orbit closure), (b) implies (c). 

Now, suppose that $\initial_{\grad^a}\widetilde I_{w,d}$ is prime. On one hand, the stabilizers $(N_{-w})_{\bm v_\la}$ and $(N_{-w}^a)_{\bm v_\la^a}$ have the same dimension (see proof of Proposition~\ref{biratprop}), and so do the orbit closures $\widetilde X_{w,d}$ and $Y_{w,d}$. On the other, denote $Z$ the variety cut out by $\initial_{\grad^a}\widetilde I_{w,d}$. By a standard construction (see, for instance,~\cite[Corollary 3.2.6]{HH}), we have a flat family over $\mathbb A^1$ with the fiber over 0 being $Z$ and with all other fibers being $\widetilde X_{w,d}$. By flatness we have $\dim Z=\dim \widetilde X_{w,d}=\dim Y_{w,d}$ but we know that $Y_{w,d}\subset Z$ which together with the irreducibility of both implies (b). 
\end{proof}

\begin{remark}\label{n45}
Condition (c) in the above theorem has been verified for $n\le 5$ and all $w$ with the use of the computer algebra system Macaulay2 (\cite{M2}). However, as will be shown below, the conditions (a), (b) and (c) may fail when $n\ge 6$. It should be pointed out that computing initial ideals and checking them for primeness appears to be rather resource intensive when $n\ge 6$. Instead, the counterexamples in Section~\ref{counterex} were found with the use of a more complicated program written in SageMath (\cite{S}). This program constructed a basis in $E_{w\la}$ incrementally by emulating the actions of the $f_{i,j}^a$ on the space $U_\la^a$ and then compared the dimensions of $E_{w\la}$ and $D_{w\la}$.
\end{remark}

Before we proceed with our counterexamples let us discuss a construction that generalizes FFLV bases and the monomial bases obtained in~\cite[Theorem 1]{fourier}. The below strategy of constructing monomial bases (or, at least, linearly independent subsets) is standard in the theory of PBW degenerations. 

Denote $\gamma$ the bijection from the set of monomials in the $f_{i,j}^a$ (spanning $\mU(\fn_{-w}^a)$) to the semigroup $\bZ_{\ge0}^{l(w)}$ which takes a monomial to its exponent vector. Choose a total monomial order $\le$ on the universal enveloping algebra, this is the same as choosing a total semigroup order (also $\le$) on the semigroup.

Now, $E_{w\om_k}\subset L_{\om_k}^a$ is spanned by vectors of the form $b_{i_1,\ldots,i_k}^a$. For every $b_{i_1,\ldots,i_k}^a\in E_{w\om_k}$ choose the $\le$-minimal monomial $M_{i_1,\ldots,i_k}\in\mU(\fn_{-w}^a)$ such that $M_{i_1,\ldots,i_k} v_\la^a$ is a nonzero multiple of $b_{i_1,\ldots,i_k}^a$. Let $\Gamma_{\om_k}$ be the set of $\gamma(M_{i_1,\ldots,i_k})$ for all $b_{i_1,\ldots,i_k}^a\in E_{w\om_k}$ (this set has size $\dim D_{w\om_k}$).

Define the Minkowski sum 
\begin{equation}\label{minksum}
\Gamma_\la=\underbrace{\Gamma_{\om_1}+\ldots+\Gamma_{\om_1}}_{a_1}+\ldots+\underbrace{\Gamma_{\om_{n-1}}+\ldots+\Gamma_{\om_{n-1}}}_{a_{n-1}}\subset\bZ_{\ge0}^{l(w)}. 
\end{equation}
And consider the set of monomials $\mathcal M_{\la}=\gamma^{-1}(\Gamma_\la)$. 

\begin{proposition}\label{linind}
The set $\{Mv_\la^a,M\in\mathcal M_\la\}$ is linearly independent. 
\end{proposition}
\begin{proof}
By Theorem~\ref{cartancomp} we are to show that the set $\mathcal M_\la u_\la^a$ is linearly independent. Consider the subspace \[S_\la=E_{\omega_1}^{\otimes a_1}\otimes\ldots\otimes E_{\omega_{n-1}}^{\otimes a_{n-1}}\subset U_\la^a,\] evidently $\mathcal M_\la u_\la^a\subset S_\la$. We may define a $\bZ_{\ge0}^{l(w)}$-grading $\hat\gamma$ on $S_\la$ by letting the grading of any tensor product of the $b_{i_1,\ldots,i_k}^a\in E_{w\omega_k}$ be the sum of the corresponding $\gamma(\mathcal M_{i_1,\ldots,i_k})$. Now, if we take $M\in\mathcal M_\la$ and expand $Mv_\la^a$ into a sum of such tensor products, the summands in the result will either lie in the $\hat\gamma$-homogeneous component of grading $\gamma(M)$ or in a component of some other grading $\delta$ with $\delta<\gamma(M)$. At least one of the former summands occurs with a nonzero coefficient and the linear independence follows.
\end{proof}

We fix a specific order which from now on will be denoted $\le$, the definitions of $\Gamma_\la$ and $\mathcal M_\la$ will now be understood with respect to this order. This order will be a graded lexicographic order with respect to a certain ordering of all $f_{i,j}^a\in\fn_{-w}^a$. This means that for two monomials we have $M<M'$ if and only if $\grad^a(M)<\grad^a(M')$ or $\grad^a(M)=\grad^a(M')$ and the first $f_{i,j}^a$ in our ordering that $M$ and $M'$ contain in unequal degrees is contained in $M$ in the lesser degree. In the ordering we choose the $f_{i,j}^a$ are sorted by $i+j$ increasing and within a fixed $i+j$ by $j$ increasing.

\begin{remark}\label{rembases}
For $w=\id$ the set $\mathcal M_\la$ provides the FFLV basis in $L_\la^a$ (\cite[Theorem 1.5]{FFL1}) and, more generally, for triangular $w$ this gives the basis constructed in~\cite[Theorem 1]{fourier}. In fact, for $n\le 5$ and all $\la$ with coordinates no greater than 3 it has been verified computationally that $|\Gamma_\la|=\dim D_{w\la}$ and, therefore, the monomials in $\mathcal M_\la$ provide a basis in $\widetilde D_{w\la}^a$. One could speculate that this always holds when $n\le 5$ but, again, we will show that this fails when $n=6$. Furthermore, in the cases we consider we will have $|\Gamma_\la|=\dim E_{w\la}$. This, however, fails for $n=6$ and $\la=(1,1,1,1,1)$.
\end{remark}

%

%


\section{Counterexamples}\label{counterex}

Within this section we set 
\begin{align*}
n&=6,\\ w&=[6,4,2,5,3,1],\\ \la&=(1,1,0,1,1),\\ \mu&=(2,1,0,1,1),\\ d=d(\la)=d(\mu)&=(1,2,4,5)
\end{align*} 
where $w$ is written in one-line notation, i.e.\ as $[w(1),\ldots,w(n)]$.

We say that two varieties with a $N_{-w}^a$-action are isomorphic as $N_{-w}^a$-varieties if their exists an $N_{-w}^a$-equivariant isomorphism between them. Our main theorem, to the proof of which we devote this section is as follows.
\begin{theorem}\label{notsame}
$X_{w\la}^a$ is not isomorphic to $X_{w\mu}^a$ as a $N_{-w}^a$-variety.
\end{theorem}
\begin{remark}
Throughout the theory of PBW degenerations the degenerate varieties are defined as orbit closures and considered together with the action of the corresponding group, therefore we consider Theorem~\ref{notsame} to be sufficient for our purposes. However, SageMath computations have shown that the number of points fixed by the action of a certain torus is different in $X_{w\la}^a$ and in $X_{w\mu}^a$ (see also Remark~\ref{fixed}). This shows that $X_{w\la}^a$ and $X_{w\mu}^a$ have different Euler characteristics and are not isomorphic as algebraic varieties and even as topological spaces.
\end{remark}

Our strategy of showing that two varieties are not isomorphic will, nonetheless, be to compare the actions of the mentioned torus, let us introduce these actions. The spaces $\widetilde D_{w\la}^a$, $\widetilde D_{w\mu}^a$, $L_\la^a$ and $L_\mu^a$ admit natural weight space decompositions, i.e.\ are equipped with an $\fh$-action and an action of the maximal torus $T_0=\exp(\fh)\subset G$. Furthermore, the grading $\grad^a$ induces an additional $\bC^*$-action by $t\in \bC^*$ acting as $t^{\grad^a}$. We obtain an action of the $n$-dimensional torus $T=T_0\times\bC^*$. 

Now, the Lie algebra $\fn_{-w}^a$ is also acted upon by $T$ (where all $\grad^a(f_{i,j}^a)=1$) and so is the Lie group $N_{-w}^a$. This means that the spaces $\widetilde D_{w\la}^a$, $\widetilde D_{w\mu}^a$, $L_\la^a$ and $L_\mu^a$ and their projectivizations are acted upon by the semidirect product $T\ltimes N_{-w}^a$. Since $T$ fixes $\widetilde{\bm v}_{\la}^a$, $\widetilde{\bm v}_{\mu}^a$ and $\bm v_\la^a$, the open $N_{-w}^a$-orbits in $X_{w\la}^a$, $X_{w\mu}^a$ and $Y_{w,d}$ are $T$-invariant and so are the orbit closures themselves.

\begin{proposition}
If $X_{w\la}^a$ and $X_{w\mu}^a$ are isomorphic as $N_{-w}^a$-varieties, then there exists a $N_{-w}^a$-equivariant isomorphism between them that is also $T$-equivariant and maps $\widetilde{\bm v}_{\la}^a$ to $\widetilde{\bm v}_{\mu}^a$.
\end{proposition}
\begin{proof}
Suppose that $\theta:X_{w\la}^a\to X_{w\mu}^a$ is an $N_{-w}^a$-equivariant isomorphism. Since each of the varieties contains a single open $N_{-w}^a$-orbit, we have $\theta(\widetilde{\bm v}_{\la}^a)=g\widetilde{\bm v}_{\mu}^a$ for some $g\in N_{-w}^a$. Now consider the map $\theta'$ with $\theta'(x)=g^{-1}\theta(x)$. Since $N_{-w}^a$ is abelian, $\theta'$ will be a $N_{-w}^a$-equivariant isomorphism mapping $\widetilde{\bm v}_{\la}^a$ to $\widetilde{\bm v}_{\mu}^a$.

Let us show that $\theta'$ is also $T$-equivariant. Indeed, the action of $T$ on $N_{-w}^a\widetilde{\bm v}_{\la}^a$ is given by $t(g'\widetilde{\bm v}_{\la}^a)=(tg't^{-1})\widetilde{\bm v}_{\la}^a$ for $t\in T$ and $g'\in N_{-w}^a$ (and $tg't^{-1}\in N_{-w}^a$). The same formula holds for $\mu$ replacing $\la$ and, since $\theta'(g'\widetilde{\bm v}_{\la}^a)=g'\widetilde{\bm v}_{\mu}^a$ for any $g'\in N_{-w}^a$, we see that $\theta'$ is $T$-equivariant when restricted to the open orbit. The claim follows.
\end{proof}

For a tuple of complex numbers $c=(c_{i,j}|f_{i,j}\in\fn_{-w})$ denote $f_c=\sum c_{i,j}f_{i,j}^a\in\fn_{-w}^a$. For a $\nu\in\Pi^+$ choose a basis in $\widetilde D_{w\nu}^a$ consisting of $T$-weight vectors and consider the coordinates of $\exp(f_c)\widetilde v_{\nu}^a$. The coordinate corresponding to a vector $v$ in our basis will be a homogeneous polynomial in the $c_{i,j}$ of degree $\grad^a(v)$. For a generic $c$ none of these coordinates vanish. 

Consider the maximal $m$ such that the $\grad^a$-homogeneous component $(\widetilde D_{w\nu}^a)_m$ is nonzero. One may define the limit of $\exp(f_{tc})\widetilde{\bm v}_{\nu}^a\in X_{w\nu}^a$ as $t\in\bC$ approaches infinity. The homogeneous coordinate of this point corresponding to a vector $v$ in the chosen bases will be zero if $\grad^a(v)<m$ it will equal the corresponding polynomial in the $c_{i,j}$ (or any $tc_{i,j}$) whenever $\grad^a(v)=m$. We denote this point $\lim_\nu(c)$. For $c$ generic in the above sense, the point $\lim_\nu(c)$ is $T$-fixed if and only if all vectors in $(\widetilde D_{w\nu}^a)_m$ have the same $T_0$-weight. We now fix a $c$ that is generic in this sense for both $\nu=\la$ and $\nu=\mu$. We will show that $\lim_\la(c)$ is $T$-fixed while $\lim_\mu(c)$ is not. 

To show that $\lim_\la(c)$ is $T$-fixed we show that $\dim((\widetilde D_{w\nu}^a)_7)=1$ while all $(\widetilde D_{w\nu}^a)_m$ with $m>7$ are zero. This will be done by identifying this one-dimensional subspace with the kernel of $\varphi_{w\la}$. The first step will be showing that $\dim(\ker\varphi_{w\la})\le 1$, i.e.\ that $\dim E_{w\la}\ge\dim\widetilde D_{w\la}^a-1$. This is where we will require a certain amount of computer assistance.

First of all, we note that $\dim\widetilde D_{w\la}^a=\dim D_{w\la}$. Tools for finding dimensions of Demazure modules can be found in LiE (\cite{LiE}) and SageMath (\cite{S}) and provide $\dim D_{w\la}=2942$.
\begin{lemma}\label{dim2941}
The dimension of $E_{w\la}$ is at least 2941.
\end{lemma}
\begin{proof}
The size of the set $\Gamma_\la$ as defined in Section~\ref{general} can be computed to equal 2941 and the lemma follows from Proposition~\ref{linind}.

For the sake of completeness we write out the relevant sets $\mathcal M_{\om_k}$ with the use of Proposition~\ref{fundamental}.
\begin{flalign*}
\mathcal M_{\om_1}=&\{1,f_{1,2}^a,f_{1,3}^a,f_{1,4}^a,f_{1,5}^a,f_{1,6}^a\},&\\
\mathcal M_{\om_2}=&\{1,f_{2,3}^a,f_{2,5}^a,f_{2,6}^a,f_{1,3}^a,f_{1,4}^a,f_{1,5}^a,f_{1,6}^a,f_{1,4}^af_{2,3}^a,f_{1,5}^af_{2,3}^a,f_{1,6}^af_{2,3}^a,f_{1,4}^af_{2,5}^a,&\\&f_{1,4}^af_{2,6}^a,f_{1,6}^af_{2,5}^a\},&\\
\mathcal M_{\om_4}=&\{1,f_{4,5}^a,f_{4,6}^a,f_{3,6}^a,f_{3,6}^af_{4,5}^a,f_{2,5}^a,f_{2,6}^a,f_{2,6}^af_{4,5}^a,f_{2,5}^af_{3,6}^a,f_{1,5}^a,f_{1,6}^a,f_{1,6}^af_{4,5}^a,&\\&f_{1,5}^af_{3,6}^a,f_{1,6}^af_{2,5}^a\},&\\
\mathcal M_{\om_5}=&\{1,f_{5,6}^a,f_{4,6}^a,f_{3,6}^a,f_{2,6}^a,f_{1,6}^a\}.&
\end{flalign*}
Herefrom the size of the set $\Gamma_\la$ is easily found to be 2941 with a computing system or programming language of choice.
\end{proof}

\begin{lemma}\label{lalemma}
We have $\dim(\ker\varphi_{w\la})=1$. This kernel coincides with $(\widetilde D_{w\nu}^a)_7$ while all $(\widetilde D_{w\nu}^a)_m$ with $m>7$ are zero.
\end{lemma}
\begin{proof}
Our first step will be showing that $\dim(\ker\varphi_{w\la})\ge 1$ which together with Lemma~\ref{dim2941} will imply the first claim. We do this by making use of Proposition~\ref{injiff}, i.e.\ by producing an element in $(J_{w,d})_\la\backslash(\initial_{\grad^a}\widetilde I_{w,d})_\la$. Explicitly, this element is \[Q=X_6 X_{4,5} X_{2,4,5,6} X_{1,3,4,5,6} - X_5 X_{4,6} X_{1,4,5,6} X_{2,3,4,5,6}.\]

The fact that $Q\in J_{w,d}$ is verified using Proposition~\ref{idealaskernel1}. One easily finds that $p^w_6=z_{1,6}$, $p^w_{4,5}=z_{1,4}z_{2,5}$, $p^w_{2,4,5,6}=-z_{1,5}z_{3,6}$, $p^w_{1,3,4,5,6}=-z_{2,6}$, $p^w_5=z_{1,5}$, $p^w_{4,6}=z_{1,4}z_{2,6}$, $p^w_{1,4,5,6}=z_{2,5}z_{3,6}$, $p^w_{2,3,4,5,6}=z_{1,6}$.

To show that $Q\notin\initial_{\grad^a}\widetilde I_{w,d}$ we make use of another ``weight'' grading on $R_d$. This $\bZ^n$-grading $\wt$ is given by the $k$th coordinate of $\wt(X_{i_1,\ldots,i_{d_k}})$ being equal to 1 if $k\in\{i_1,\ldots,i_{d_k}\}$ and 0 otherwise. Note that such a grading can be viewed as a $(\bC^*)^n$-action on $\bP$. The image of the diagonal embedding $\bC^*\in(\bC^*)^n$ then acts trivially and we obtain an action of $(\bC^*)^n/\bC^*$ which is the action of the maximal torus $T_0$. Hence, the ideals we consider are $\wt$-homogeneous.

Now, the ideal $\widetilde I_{w,d}$ is known (see, for instance,~\cite[Remark 9]{KM}) to be generated by $I_d$ and the variables $X_{i_1,\ldots,i_{d_k}}$ for which there is an $1\le l\le d_k$ such that the $l$th smallest number among $w(i_1),\ldots,w(i_k)$ is greater than the $l$th smallest number among $w(1),\ldots,w(k)$. (In other words, such $X_{i_1,\ldots,i_{d_k}}$ that $b_{i_1,\ldots,i_{d_k}}\notin \widetilde D_{w\om_{d_k}}$). In our case we obtain two such variables: $X_{1,4}$ and $X_{1,2,4,5}$. 

However, it is easily checked that there are no monomials in $(R_d)_\la$ that have the same $\wt$-grading as $Q$ and are divisible by either $X_{1,4}$ or $X_{1,2,4,5}$. This means that if we have $Q\in\initial_{\grad^a}\widetilde I_{w,d}$, then we must have $Q\in\initial_{\grad^a}\widetilde I_d$. But the latter is disproved with the use of Proposition~\ref{idealaskernel}. We have $p_{4,5}=z_{1,4}z_{2,5}-z_{1,5}z_{2,4}$ which divides none of $p_5$, $p_{4,6}$, $p_{1,4,5,6}$ and $p_{2,3,4,5,6}$.

We have shown that the kernel is one-dimensional and we let $q$ be a spanning vector. 

Now, since $\grad^a(Q)=6$ we see that $\dim((J_{w,d})_\la)_6>\dim((\initial_{\grad^a}\widetilde I_{w,d})_\la)_6$. In fact, Propositions~\ref{edual} and~\ref{hdual} imply that \[\dim(J_{w,d})_\la-\dim(\initial_{\grad^a}\widetilde I_{w,d})_\la=\dim H_{w\la}-\dim E_{w\la}=\dim\widetilde D_{w\la}-\dim E_{w\la}=1.\] Hence \[\dim((J_{w,d})_\la)_6-\dim((\initial_{\grad^a}\widetilde I_{w,d})_\la)_6=1\] while $\dim((J_{w,d})_\la)_m=\dim((\initial_{\grad^a}\widetilde I_{w,d})_\la)_m$ for all other $m$. Consequently, \\$\dim(E_{w\la})_m<\dim(H_{w\la})_m$ if and only if $m=6$.

Denote $\grad^a(q)=r$. By the definition of $\varphi_{w\la}$, the fact that $\varphi_{w\la}(q)=0$ means that the component $(\widetilde D_{w\la})_{\le r-1}$ of the PBW filtration with respect to $\fn_{-w}$ is strictly smaller than $\widetilde D_{w\la}\cap(L_\la)_{\le r-1}$ (component of the PBW filtration with respect to $\fn_-$). The fact that $\varphi_{w\la}$ is injective when restricted to any $(\widetilde D_{w\la}^a)_m$ with $m\neq r$ means that $(\widetilde D_{w\la})_{\le m-1}=\widetilde D_{w\la}\cap(L_\la)_{\le m-1}$. Setting $m=r-1$ we obtain 
\begin{multline*}
\dim(E_{w\la})_{r-1}=\dim(\widetilde D_{w\la}^a)_{r-1}=\dim(\widetilde D_{w\la})_{\le r-1}-\dim(\widetilde D_{w\la})_{\le r-2}<\\\dim(\widetilde D_{w\la}\cap(L_\la)_{\le r-1})-\dim(\widetilde D_{w\la}\cap(L_\la)_{\le r-2})=\dim(H_{w\la})_{r-1}.
\end{multline*}
In view of the above this immediately provides $r-1=6$.

Finally, observe that the $(L_\la^a)_m=0$ and hence $(E_{w\la})_m=0$ whenever $m>6$. This is since the highest $\grad^a$ grading occurring in $L_\nu^a$ is additive with respect to $\nu$ (due, for instance, to the Minkowski sum property of the FFLV polytopes, see~\cite[Proposition 3.7]{FFL1}) and for $L_{\om_k}^a$ this highest grading is $\min(k,n-k)$ (since that is the maximal value of the right-hand side of~(\ref{grada})). Therefore, for any $v\in\widetilde D_{w\la}^a$ with $\grad^a(v)>6$ we have $\varphi_{w\la}(v)=0$, i.e.\ $q$ and its multiples are the only such $v$.
\end{proof}

Now let us see what can be said about $\widetilde D_{w\mu}^a$.

\begin{lemma}\label{mulemma}
We have $\dim(\ker\varphi_{w\mu})=5$. This kernel coincides with $(\widetilde D_{w\mu}^a)_8$ and is spanned by 5 vectors with pairwise distinct $T_0$-weights.
\end{lemma}
\begin{proof}
First, in the spirit of the proof of Lemma~\ref{dim2941} the size of the set $\Gamma_\mu$ is computed to equal 8221. One may further compute that $\dim D_{w\mu}=8226$. We see that $\dim E_{w\mu}\ge 8221$ and $\dim(\ker\varphi_{w\la})\le 5$.

Replacing the vector $q\in\widetilde D_{w\la}^a$ with a scalar multiple of itself if necessary, we may choose a monomial $M\in\mU(\fn_{-w}^a)$ of degree 7 such that $M\widetilde v_{w\la}^a=q$. To prove the lemma it now suffices to show that the vectors $f_{1,2}^aM\widetilde v_{w\mu}^a$, $f_{1,3}^aM\widetilde v_{w\mu}^a$, $f_{1,4}^aM\widetilde v_{w\mu}^a$, $f_{1,5}^aM\widetilde v_{w\mu}^a$ and $f_{1,6}^aM\widetilde v_{w\mu}^a$ are nonzero and lie in $\ker\varphi_{w\mu}$. 

Choose a $2\le j\le 6$. First we show that $f_{1,j}^aM\widetilde v_{w\mu}^a$ is nonzero. Choose a monomial in $M_0\in\mU(\fn_{-w})$ obtained from $M$ by removing the ${}^a$ superscipts and consider the product of $\fn_{-w}$-modules $\widetilde D_{w\om_1}\otimes\widetilde D_{w\la}$. The cyclic submodule therein generated by $v_{\om_1}\otimes v_\la$ is isomorphic to $\widetilde D_\mu$. We are to show that $f_{1,j}M_0(v_{\om_1}\otimes v_\la)$ may not be expressed as a linear combination of vectors of the form $M'(v_{\om_1}\otimes v_\la)$ with $M'\in\mU(\fn_{-w})_{\le 7}$. Let us consider a basis in $\widetilde D_{w\la}$ comprised of $M_0v_\la$ and a basis in $\mU(\fn_{-w})_{\le 6}v_\la$. Together with the basis consisting of $e_k$ in $\widetilde D_{w\om_1}$ this provides a basis in the tensor product. The decomposition of $f_{1,j}M_0(v_{\om_1}\otimes v_\la)$ with respect to our basis is then seen to contain $e_j\otimes M_0v_\la$ with a nonzero coefficient. However, the decomposition of every $M'(v_{\om_1}\otimes v_\la)$ with $M'\in\mU(\fn_{-w})_{\le 7}$ contains $e_j\otimes M_0v_\la$ with a zero coefficient.

The fact that $f_{1,j}^aM\widetilde v_{w\mu}^a\in\ker\varphi_{w\mu}$ simply follows from $(E_{w\mu})_8=(L_\mu^a)_8=0$.
\end{proof}

Theorem~\ref{notsame} is now immediate.
\begin{proof}[Proof of Theorem~\ref{notsame}]
The discussion preceding Lemma~\ref{dim2941} together with Lemmas~\ref{lalemma} and~\ref{mulemma} shows that $\lim_\la(c)$ is $T$-fixed while $\lim_\mu(c)$ is not.
\end{proof}

\begin{remark}\label{fixed}
Although not immediate from the above, it is true that neither of $X_{w\la}^a$ and $X_{w\mu}^a$ coincides with $Y_{w,d}$. In fact, the rational map $\bm\varphi_{w\la}$ establishes a bijection between the $T$-fixed points in $X_{w\la}^a\backslash\bP(\ker\varphi_{w\la})$ and the $T$-fixed points in $Y_{w,d}$. The same holds for the map $\bm\varphi_{w\mu}$ and the $T$-fixed points in $X_{w\mu}^a\backslash\bP(\ker\varphi_{w\mu})$. This means that if the number of $T$-fixed points in $Y_{w,d}$ is $m$, then the number of $T$-fixed points in $X_{w\la}^a$ is $m+1$ and the number of $T$-fixed points in $X_{w\mu}^a$ is $m+5$. However, proving this requires substantially more computer assistance, for this reason the above argument was chosen.
\end{remark}

\begin{remark}
We have seen that the existing definition of abelian PBW degenerations of Schubert varieties has a significant disadvantage: these varieties do not depend on the highest weight the way one would expect, i.e.\ the way Schubert varieties do themselves. However, this dependence may still be worth an investigation. In particular, one fairly natural and interesting question arises: might it be that if we fix $d(\la)$ while letting $\la$ approach infinity in some sense, then $X_{w\la}^a$ will stabilize letting us define the PBW degeneration as this limit form?
\end{remark}

\end{document}